\documentclass[a4paper,11pt]{amsart}

\usepackage{comment}

\usepackage[hang,small,bf]{caption}
\usepackage[subrefformat=parens]{subcaption}
\usepackage{etoolbox}
\ifdef{\crop}{%
\usepackage[includeheadfoot,twoside=False,paperwidth=448pt,paperheight=587pt,rmargin=15pt,lmargin=15pt,tmargin=15pt,bmargin=15pt]{geometry}%
}{%
\setlength{\topmargin}{22mm}
\addtolength{\topmargin}{-1in}
\setlength{\oddsidemargin}{27mm}
\addtolength{\oddsidemargin}{-1in}
\setlength{\evensidemargin}{27mm}
\addtolength{\evensidemargin}{-1in}
\setlength{\textwidth}{156mm}
\setlength{\textheight}{240mm}
}%

\usepackage{mathtools}
\usepackage{color}
\usepackage[active]{srcltx}
\usepackage{amsmath,amsthm,amsxtra}
\usepackage{amssymb}

\usepackage[mathscr]{eucal}

\usepackage{bm}
\usepackage[all]{xy}

\usepackage{aliascnt}

\usepackage{tabularx}
\newcolumntype{C}{>{\centering\arraybackslash}X} 

\theoremstyle{plain}
\newtheorem{thm}{Theorem}[section]
\newtheorem*{thm*}{Theorem}

\newaliascnt{prop}{thm}
\newaliascnt{cor}{thm}
\newaliascnt{lem}{thm}
\newaliascnt{claim}{thm}
\newaliascnt{defn}{thm}
\newaliascnt{ques}{thm}
\newaliascnt{conj}{thm}
\newaliascnt{fact}{thm}
\newaliascnt{rem}{thm}
\newaliascnt{ex}{thm}
\newaliascnt{sett}{thm}
\newtheorem{prop}[prop]{Proposition}
\newtheorem{cor}[cor]{Corollary}
\newtheorem{lem}[lem]{Lemma}
\newtheorem{claim}[claim]{Claim}
\newtheorem*{prop*}{Proposition}
\newtheorem*{cor*}{Corollary}
\newtheorem*{lem*}{Lemma}
\newtheorem*{claim*}{Claim}
\theoremstyle{definition}
\newtheorem{defn}[defn]{Definition}

\newtheorem*{defn*}{Definition}
\newtheorem*{ques*}{Question}
\newtheorem*{conj*}{Conjecture}

\newtheorem*{prob*}{Problem}

\newtheorem{rem}[rem]{Remark}
\newtheorem{ex}[ex]{Example}
\newtheorem*{fact*}{Fact}
\newtheorem*{rem*}{Remark}
\newtheorem*{ex*}{Example}
\aliascntresetthe{prop}
\aliascntresetthe{cor}
\aliascntresetthe{lem}
\aliascntresetthe{claim}
\aliascntresetthe{defn}
\aliascntresetthe{ques}
\aliascntresetthe{conj}
\aliascntresetthe{fact}
\aliascntresetthe{rem}
\aliascntresetthe{ex}
\aliascntresetthe{sett}

\usepackage[pointedenum]{paralist}
\usepackage{varioref}
\labelformat{equation}{\textnormal{(#1)}}
\labelformat{enumi}{\textnormal{(#1)}}

\usepackage[a4paper]{hyperref}

\def\textsectionN~{\textsection{}}

\renewcommand\phi{\varphi}
\renewcommand\epsilon{\varepsilon}
\renewcommand\leq{\leqslant}
\renewcommand\geq{\geqslant}

\makeatletter
\newcommand{\set}{  \@ifstar{\@setstar}{\@set}}\newcommand{\@setstar}[2]{\{\, #1 \, ; \,  #2 \,\}}
\newcommand{\@set}[1]{\{ #1 \}}
\makeatother

\newcommand{\trans}[1][1]{\raisebox{#1ex}{\scriptsize\kern0.1em$t$\kern-0.1em}}

\DeclareMathOperator{\Hom}{Hom}
\DeclareMathOperator{\im}{im}

\DeclareMathOperator{\Pic}{Pic}

\DeclareMathOperator{\Supp}{Supp}

\DeclareMathOperator{\rank}{rank}

\DeclareMathOperator{\coker}{coker}

\DeclareMathOperator{\Ext}{Ext}
\DeclareMathOperator{\ch}{ch}
\DeclareMathOperator{\NS}{NS}
\DeclareMathOperator{\Coh}{Coh}

\def\Z{\mathbb{Z}}
\def\Q{\mathbb{Q}}
\def\R{\mathbb{R}}
\def\C{\mathbb{C}}

\def\r+{\mathbb{R}_{\geq 0}}

\def\ep{\varepsilon}

\def\r+{{\R}_{\geq 0}}
\def\q+{{\Q}_{\geq 0}}
\def\P{\mathbb{P}}
\def\*c{\C^{\times}}

\def\<{\langle}
\def\>{\rangle}

\def\op1{\calo_{\P^1}}

\def\C{\mathbb {C}}

\def\Q{\mathbb {Q}}
\def\R{\mathbb {R}}

\def\Z{\mathbb {Z}}

\newcommand{\calf}{\mathcal {F}}

\newcommand{\calh}{\mathcal {H}}
\newcommand{\cali}{\mathcal {I}}

\newcommand{\calo}{\mathcal {O}}

\newcommand{\calt}{\mathcal {T}}

\newcommand{\ul}{\underline{l}}
\newcommand{\un}{\underline{n}}
\newcommand{\uh}{\underline{h}}

\makeatletter
  
  \@addtoreset{equation}{section}
\makeatother

\title
[Remarks on basepoint-freeness thresholds of polarized  abelian surfaces]
{Remarks on basepoint-freeness thresholds of polarized abelian surfaces}

\author[A.~Ito]{Atsushi~Ito}
\address{Department of Mathematics, Institute of Pure and Applied Sciences, University of Tsukuba, Tsukuba, Ibaraki 305-8571, Japan}
\email{ito-atsushi@math.tsukuba.ac.jp}


\begin{document}

\begin{abstract}
We determine the basepoint-freeness threshold of a very general polarized abelian surface over the field of complex numbers. 
We also give the first example of a polarized abelian surface whose basepoint-freeness threshold is irrational.
\end{abstract}

\maketitle

\section{Introduction}

Throughout this paper, 
we work over the field of complex numbers $\C$.
Let $(X,\underline{l})$ be a polarized abelian variety.
In \cite{MR4157109}, Jiang and Pareschi introduce the \emph{basepoint-freeness threshold} $\beta(X,\underline{l})$.
This invariant turns out to be a powerful tool in the study of linear series on abelian varieties, such as syzygies and jet separation \cite{MR4157109, MR4114062, MR4781959, MR4474904}.
While several bounds on  $\beta(X,\underline{l})$ have been established under various geometric assumptions
(e.g.\ \cite{MR4518246, MR4626890, MR4689554, MR4395094, MR4589526, MR4363833}),
there are still only a few examples where the exact value is known.

When $X$ is an abelian surface, building on earlier joint work with Lahoz \cite{MR4579982},
Rojas \cite{MR4363833} investigates $\beta(X, \underline{l})$ using Bridgeland stability conditions and obtains the following result.

\begin{thm}[{part of \cite[Theorem A (3)]{MR4363833}}]\label{thm_Rojas}
Let $(X,\underline{l})$ be a polarized abelian surface of type $(1,d)$ such that the intersection number $(\underline{l} \cdot D)$ is divisible by $2d$ for every divisor $D$ on $X$.
If $d$ is not a perfect square,
then $\beta(X,\underline{l}) \in  \left\{ \frac{2y_0}{x_0-1}, \frac{2y_1}{x_1-1} \right\}$, where $(x_0,y_0)$ and $(x_1,y_1)$  are the minimal and the second minimal positive solutions to  Pell's equation $x^2-4d y^2=1$, respectively.
In particular,
$\beta(X,\underline{l}) \leq  \frac{2y_0}{x_0-1}$.
\end{thm}

\begin{rem}
\begin{enumerate}
\item If $d$ is a perfect square, then $\beta(X,\underline{l}) =\frac{1}{\sqrt{d}}$ by \cite[Theorem A (1)]{MR4363833}.
\item The condition that $2d \mid (\underline{l} \cdot D) $ for every divisor $D $ holds when the Picard number of $X$ is one.
\end{enumerate}
\end{rem}

The main result of this paper gives the exact value of $\beta(X,\underline{l})$ in this setting.
More generally, we determine $\beta(X,\underline{n})$ for a polarization $\underline{n}$ that is sufficiently close to being proportional to $\underline{l}$ as follows.

\begin{thm}\label{thm_BFT_surface}
Under the assumptions of Theorem \ref{thm_Rojas}, 
let $\underline{n}$ be a polarization on $X$, possibly equal to $\underline{l}$.
If $x_0^2 (\underline{n}^2) \geq 2y_0^2 (\underline{l} \cdot \underline{n})^2$, then
\begin{align*}
\beta(X,\underline{n}) =  \frac{2y_0}{x_0-1} \cdot  \frac{(\underline{l} \cdot \underline{n} ) - \sqrt{(\underline{l} \cdot \underline{n})^2 - (\underline{l}^2) (\underline{n}^2)}}{(\underline{n}^2)} .
\end{align*}
In particular, 
$\beta(X,\underline{l}) =\frac{2y_0}{x_0-1}$.
\end{thm}

The proof is outlined as follows.
First, by combining the lower bound for $\beta(X, \underline{l})$ provided by Alvarado--Pareschi \cite{Alvarado:2024aa} with Theorem \ref{thm_Rojas},
we obtain the equality $\beta(X, \underline{l}) = \frac{2y_0}{x_0-1}$.
Following the arguments in \cite{MR4579982} and \cite{MR4363833}, 
this equality induces a resolution 
\[
0 \to E \to F \to \cali_o \to 0
\]
of the ideal sheaf $\cali_o$ of the origin $o \in X$ by semihomogeneous vector bundles,
whose Chern characters can be computed explicitly.
We show that this resolution determines the exact value of $\beta(X, \underline{n})$
if the numerical condition $x_0^2 (\underline{n}^2) \geq 2y_0^2 (\underline{l} \cdot \underline{n})^2$ is satisfied.
It should be noted that the lower bound in \cite{Alvarado:2024aa} also essentially relies on the properties of semihomogeneous bundles.

\vspace{2mm}

Theorem \ref{thm_BFT_surface} allows us to construct polarized abelian surfaces with irrational basepoint-freeness thresholds. For instance, we obtain the following corollary.

\begin{cor}\label{cor_irrational_Intro}
There exists a polarized abelian surface $(X,\underline{n})$ of type $(1,17)$ such that $\beta(X,\underline{n}) = \frac{6-\sqrt{2}}{17}$. 
\end{cor}

We comment on  Theorem \ref{thm_BFT_surface} and Corollary \ref{cor_irrational_Intro} in relation to the Seshadri constant $\ep(X,\underline{l})$,
another invariant measuring the positivity of the polarization $\underline{l}$.
In \cite[\S 3]{MR4518246},
we observe some similarities between $\beta(X,\underline{l})$ and $\ep(X,\underline{l})$.
A similar phenomenon appears in Theorem \ref{thm_BFT_surface},
as Bauer \cite[Theorem 6.1]{MR1678549} determines $\ep(X,\underline{l})$ using the minimal solution to Pell's equation $x^2-2d y^2=1$ for abelian surfaces with Picard number one.

Seshadri constants are defined for any polarized variety, and it is a famous open problem whether they can be irrational.
In the case of abelian surfaces,
Bauer and Szemberg  \cite[Corollary A.2]{MR1660259} show that $\ep(X,\underline{l})$ is always rational. 
On the other hand, Caucci suggested that basepoint-freeness thresholds of polarized abelian varieties  could be irrational,
 even though no such example was known at the time \cite[Remark 2.3]{MR4114062}.
Following Theorem \ref{thm_Rojas}, Rojas raised the question of whether 
$\beta(X,\underline{l})$  is rational for every polarized abelian surface \cite[Remark 4.2]{MR4363833}.
Corollary \ref{cor_irrational_Intro} provides the first example of a polarized abelian surface with irrational basepoint-freeness threshold.

After completing the draft of this paper, the author was informed through private communication that Chunyi Li, Lei Song, and Xiao Wang had also independently discovered the example in Corollary \ref{cor_irrational_Intro}.

\vspace{2mm}
This paper is organized as follows.
In \S \ref{sec_preliminaries}, we recall some notation on cohomological rank functions,  basepoint-freeness thresholds, and semihomogeneous bundles.
In \S \ref{sec_beta=}, we determine $\beta(X,\underline{l})$ under the assumption of Theorem \ref{thm_Rojas}.
In \S \ref{sec_proof} and \S \ref{sec_proof_cor}, we give proofs of Theorem \ref{thm_BFT_surface} and Corollary \ref{cor_irrational_Intro}, respectively.

\subsection*{Acknowledgments}
The author would like to thank Professors  Mart\'i Lahoz and Andr\'es Rojas for their useful comments and suggestions.
The author was supported by JSPS KAKENHI Grant Numbers 21K03201, 26K06719.

\section{Preliminaries}\label{sec_preliminaries}

\subsection{Cohomological rank functions and basepoint-freeness thresholds}\label{subsec_BFT}

Let $(X,\underline{l})$ be a polarized abelian variety and let $F$ be a coherent sheaf on $X$.
In \cite{MR4157109}, Jiang and Pareschi introduce the \emph{cohomological rank function} $ h^i_{F,\underline{l}}  \colon \Q \to \Q_{\geq 0}$ as
\[
h^i_{F,\underline{l}} (x) \coloneqq b^{-2g} h^i(X, \mu_b^* F \otimes L^{ab} \otimes P_\alpha) 
\]
for $x=\frac{a}{b}$,
where $ \mu_b   \colon  X \to X$ is the multiplication-by-$b$ isogeny, $L$ is an ample line bundle representing $\underline{l}$,
and  $P_\alpha$ is the line bundle on $X$ corresponding to  a general point $\alpha$ of the dual abelian variety $  \widehat{X}=\Pic^0(X)$.

\begin{defn}[{\cite[Section 8]{MR4157109}}]\label{def_Bpf_threshold}
Let $(X,\underline{l})$ be a polarized abelian variety.
The \emph{basepoint-freeness threshold} $\beta(X,\underline{l})$ is defined by
\[
\beta(X,\underline{l}) \coloneqq \inf \{ x \in \Q \, | \,  h^1_{\cali_o,\underline{l}}(x) =0 \},
\]
where $\cali_o$ is the ideal sheaf corresponding to the origin $o \in X$.
\end{defn}

In \cite[\S 3.1]{Alvarado:2024aa}, Alvarado and Pareschi introduce a related invariant 
\[
\beta^0(X,\underline{l})\coloneqq \sup \{ x \in \Q \, | \,  h^0_{\cali_o,\underline{l}}(x) =0 \}.
\]
We note that $\beta(X,\underline{l})$ and $\beta^0(X,\underline{l})$ are denoted by $\beta^1_X(\underline{l})$ and $\beta^0_X(\underline{l})$ respectively in \cite{Alvarado:2024aa}.

\subsection{Semihomogeneous bundles}

A vector bundle $E$ on an abelian variety $X$ is called \emph{semihomogeneous} if 
for each $x \in X$, there exists $\alpha \in \widehat{X}$ such that $t_x^* E \simeq E \otimes P_\alpha$,
where $t_x   \colon  X \to X$ is the translation by $x$.
We refer the reader to  \cite{MR498572} and \cite[\S 1.5]{Alvarado:2024aa} for details.

\begin{enumerate}
\item \cite[Proposition 6.17, Theorem 7.11 (1)]{MR498572} 
For $\lambda \in \Q$ and a class $\underline{l} \in \NS(X) $ in the N\'eron--Severi group,
there exists a simple semihomogeneous bundle $E_{\lambda \underline{l} } $ with $c_1(E_{\lambda \underline{l} } )/\rank E_{\lambda \underline{l} }  = \lambda \underline{l}  \in \NS(X)_\Q$.
Such $E_{\lambda \underline{l} } $ is unique up to tensoring $P_\alpha$ for $\alpha \in \widehat{X}$.
\item \cite[Theorem 7.11 (5)]{MR498572}
For $\lambda=\frac{a}b \in \Q$ with $b >0$,
let $X[b]$ denote the group of $b$-torsion points on $X$ and 
set
\[
 u_{\underline{l}}(a,b) =\sqrt{ \left|X[b] \cap K(a \underline{l}) \right| }, 
\]
where $ K(a \underline{l})$ is the kernel of $\varphi_{a \underline{l} }   \colon  X \to \widehat{X}$ induced by the class $a \underline{l}$.
Then
\begin{align}\label{eq_rank,chi}
\rank E_{\lambda \underline{l}} =  r_{\underline{l}}(\lambda) \coloneqq  \frac{b^g}{u_{\underline{l}}(a,b) }, 
\quad \chi(E_{\lambda \underline{l}}) = \frac{a^g \chi(\underline{l})}{u_{\underline{l}}(a,b) } = r_{\underline{l}}(\lambda) \frac{(\lambda \underline{l})^g}{g!},
\end{align}
where $g=\dim X$.
\item  \cite[Proposition 6.16]{MR498572} For any $\lambda$ and $ \underline{l}$,
$E_{\lambda \underline{l}}$ is Gieseker stable with respect to any ample line bundle $H$.
If $\dim X=2$, $E_{\lambda \underline{l}}$ is $\mu_H$-stable
since the discriminant $\Delta(E_{\lambda \underline{l}}) = 2 (\rank E_{\lambda \underline{l}}) c_2(E_{\lambda \underline{l}}) - (\rank E_{\lambda \underline{l}} -1) c_1(E_{\lambda \underline{l}})^2 = \ch_1(E_{\lambda \underline{l}})^2 -2 \ch_0( E_{\lambda \underline{l}})  \cdot \ch_2(E_{\lambda \underline{l}} )  $ is zero by (2).
\end{enumerate}

For a coherent sheaf $F$ on $X$, set
\[
v(F) \coloneqq  (\ch_0(F),\ch_1(F),\ch_2(F)) = (\rank F, c_1(F), \chi(F)) \in \Z \oplus \NS (X) \oplus \Z.
\]

\begin{lem}\label{lem_rank_chi}
Let $(X,\underline{l})$ be a polarized abelian surface of type $(1,d)$.
Assume that $d$ is not a perfect square, and let $(x_0,y_0)$ be  the minimal positive solution to  Pell's equation $x^2-4d y^2=1$.
Then
\begin{align*}
v\Big(E_{\frac{2y_0}{x_0\pm1} \underline{l}} \Big) = \left( \frac{x_0\pm 1}{2} , \,  y_0 \underline{l}, \, \frac{x_0\mp 1}{2}   \right).
\end{align*}
\end{lem}

\begin{proof}
It follows from the Pell equation that $x_0$ is odd and
\begin{align}\label{eq_pell}
\frac{x_0-1}2 \cdot \frac{x_0+1}2=dy_0^2.
\end{align}

\begin{claim}\label{claim_GCD}
Set $A_\pm=\gcd (\frac{x_0\pm1}2,y_0)$ and  $B_\pm=\gcd (\frac{x_0\pm1}2,d y_0)$, where $\gcd (a,b)$ denotes the greatest common divisor of $a$ and $b$.
For a prime number $p$, let $v_p$ denote the $p$-adic valuation.
Then
$v_p(A_\pm) +v_p(B_\pm)=v_p(\frac{x_0\pm1}{2})$.
\end{claim}

\begin{proof}
We only prove this claim for  $A_+,B_+ $.
Since this claim is trivial when $p \nmid \frac{x_0+1}2$, we  may assume $p \mid \frac{x_0+1}2$.
Then $ p \nmid \frac{x_0-1}2=\frac{x_0+1}2-1$ and hence
\begin{align}\label{eq_pell2}
v_p \left(\frac{x_0+1}2 \right) =v_p \left(\frac{x_0-1}2 \cdot \frac{x_0+1}2 \right) =v_p(dy_0^2) = v_p(y_0) + v_p(dy_0)
\end{align}
by \ref{eq_pell}.
In particular, both $ v_p(y_0)$ and $ v_p(dy_0)$ are at most $v_p(\frac{x_0+1}2) $. 
Hence  $v_p(A_+)= v_p(y_0), v_p(B_+)=v_p(dy_0)$, and thus this claim follows from \ref{eq_pell2}.
\end{proof}

Recall that 
\[
 u_{\underline{l}}\left(y_0,\frac{x_0\pm 1}2\right) =\sqrt{ \left|X\left[\frac{x_0\pm 1}2\right] \cap K(y_0 \underline{l}) \right| }.
\]
Since 
$K(y_0 \underline{l})$ is isomorphic to $(\Z/ y_0\Z \oplus \Z/dy_0\Z)^{\oplus 2}$,  we have
\begin{align*}
u_{\underline{l}}\left(y_0,\frac{x_0\pm 1}2\right) &=\left| \left\{ z \in  \Z/ y_0\Z \oplus \Z/dy_0\Z \  \Big| \ \frac{x_0\pm 1}2 \cdot z=0  \right\} \right| \\
&= \gcd \left(\frac{x_0\pm 1}2, y_0\right) \cdot \gcd \left(\frac{x_0\pm 1}2,d y_0\right) =\frac{x_0\pm 1}2,
\end{align*}
where the last equality holds by Claim \ref{claim_GCD}.
Then this lemma holds by \ref{eq_rank,chi} and $c_1( E_{\lambda \underline{l}}) =\rank  E_{\lambda \underline{l}} \cdot  \lambda \underline{l}$.
\end{proof}

\section{Computation of $\beta(X,\underline{l})$ in Theorem \ref{thm_Rojas}}\label{sec_beta=}

Using a result in \cite{Alvarado:2024aa}, we determine the exact value of $\beta(X,\underline{l}) $ in Theorem \ref{thm_Rojas} as follows.

\begin{prop}\label{prop_Rojas}
In the setting of Theorem \ref{thm_Rojas}, $\beta(X,\underline{l}) =\frac{2y_0}{x_0-1}$ holds.
\end{prop}

\begin{proof}
Let $d$ be a positive integer that is not a perfect square, and let $(x_0,y_0)$ be  the minimal positive solution to  Pell's equation $x^2-4d y^2=1$.
Following the notation in \cite{Alvarado:2024aa},
let $\underline{\hat{l}}$ be the dual polarization of $\underline{l} $ on $\widehat{X}$.
By \cite[Theorem D and (1.9)]{Alvarado:2024aa}, we have
\begin{align*}
\beta(X,\underline{l}) \geq \sup_{\nu > \frac1{\sqrt{\chi(\underline{\hat{l}})}} } \frac{1+ r_{\underline{\hat{l}}}(\nu)}{d  \nu   r_{\underline{\hat{l}}}(\nu)},
\end{align*}
where $\nu $ in the supremum is a rational number and $r_{\underline{\hat{l}}}(\nu) =\rank E_{\nu \underline{\hat{l}}} $ defined by \ref{eq_rank,chi}.
Since the type of $\underline{\hat{l}}$ is also $(1,d)$,
we can take $\nu = \frac{2y_0}{x_0-1} > \frac{1}{\sqrt{d}} =\frac1{\sqrt{\chi(\underline{\hat{l}})}} $.
Then we have
\begin{align}\label{eq_rank}
r_{\underline{\hat{l}}}\left( \frac{2y_0}{x_0-1} \right)   =\frac{x_0-1}{2}
\end{align}
by Lemma \ref{lem_rank_chi}.
Hence it holds that 
\begin{align*}
\beta(X,\underline{l}) \geq \frac{1+\frac{x_0-1}{2} }{ d \cdot  \frac{2y_0}{x_0-1} \cdot  \frac{x_0-1}{2}} = \frac{x_0+1}{2dy_0} =\frac{2y_0}{x_0-1}.
\end{align*}
Since $\beta(X,\underline{l}) \leq \frac{2y_0}{x_0-1}$ by Theorem \ref{thm_Rojas},
this completes the proof.
\end{proof}

\begin{rem}
Since Theorem \ref{thm_Rojas}, \cite[Theorem D and (1.9)]{Alvarado:2024aa}, and Lemma \ref{lem_rank_chi} hold over any algebraically closed field of characteristic zero,
so does Proposition \ref{prop_Rojas}.
\end{rem}

\begin{ex}
For certain values of $d$, we can write down the minimal solution $(x_0,y_0)$ explicitly as follows (see \cite[Theorem 3.2.1]{MR1383823} for example):
\begin{enumerate}
\setlength{\itemsep}{1mm}
\item Let $m, k$ be positive integers with $k \mid m$.
If $d=m^2 \pm k$, then $\left(\frac{2m^2}k \pm 1, \frac{m}k \right)$ is the minimal solution to $x^2-4dy^2=1$.
Hence for $(X,\underline{l})$ in Theorem \ref{thm_Rojas}, we have
\[
\beta(X,\underline{l}) = 
\begin{cases}
\frac{1}{m}  &  \text{ if } d=m^2+k \\[1mm]
\frac{m}{d}  &   \text{ if } d=m^2-k .
\end{cases}
\]
This generalizes \cite[Lemmas A.1,A.4]{MR4626890}.
\item Let  $m, k$ be positive integers with $4k \mp 1 \mid 2m \pm 1$, and set $s =\frac{2m \pm 1}{4k \mp 1}$.
If $d=m^2 \pm (m +  k)$, then $(2 (2m \pm 1)s \pm 1, 2s)$ is the minimal solution to $x^2-4dy^2=1$.
Hence for $(X,\underline{l})$ in Theorem \ref{thm_Rojas}, we have
\[
\beta(X,\underline{l}) = 
\begin{cases}
\frac{2}{2m+1}  &  \text{ if } d=m^2+m+k \\[1mm]
\frac{2m-1}{2d}  &   \text{ if } d=m^2 -m-k  .
\end{cases}
\]
\end{enumerate}
\end{ex}

\begin{rem}\label{remark_coh_rank_function}
For $(X,\underline{l})$ in Theorem \ref{thm_Rojas},
we have 
\begin{align*}
h^0_{\cali_o,\underline{l}} (x) = 
\begin{cases}
0 &  \text{ if  } x \leq \frac{2y_0}{x_0+1}  \\[1mm]
\frac{d(x_0+1)}{2} x^2- 2dy_0 x + \frac{x_0-1}{2} &  \text{ if  } \frac{2y_0}{x_0+1}  \leq x \leq \frac{2y_0}{x_0-1}  \\[1mm]
dx^2-1  &  \text{ if  } x \geq  \frac{2y_0}{x_0-1} 
\end{cases} 
\end{align*}
by \cite[Theorem A (2)]{MR4363833} and Proposition \ref{prop_Rojas}.
In particular, $\beta^0(X,\underline{l}) = \frac{2y_0}{x_0+1} $.
\end{rem}

\section{Proof of Theorem \ref{thm_BFT_surface}}\label{sec_proof}

Throughout this section, $\underline{l}=[L]$ and $\underline{n}=[N]$ are polarizations on an abelian surface $X$, possibly $\underline{l}=\underline{n}$.
Assume that $\underline{l}$ is of type $(1,d)$ with non-perfect square $d$.
Let $(x_0,y_0)$ be the minimal solution to $x^2-4d y^2=1$.

\subsection{Resolutions of $\cali_o$ by semihomogeneous bundles}

Let $F$ be a coherent sheaf on $X$ and $\chi_{F,\underline{n}} (x) $ be the Hilbert polynomial of $F$ with respect to $\underline{n}$.
Then 
\[
\chi_{F,\underline{n}} (x) = h^0_{F,\underline{n}} (x) - h^1_{F,\underline{n}} (x)  + h^2_{F,\underline{n}} (x)  
\]
holds for any $ x \in \Q$,
where $h^i_{F,\underline{n}} (x)  $ is the cohomological rank function in \S \ref{subsec_BFT}.
If $F=\cali_o$ and $x \geq 0$,  $h^2_{\cali_o,\underline{n}} (x)  =0$ and hence $\chi_{\cali_o,\underline{n}} (x) = h^0_{\cali_o,\underline{n}} (x) - h^1_{\cali_o,\underline{n}} (x) $ holds.

\begin{lem}\label{lem_h^i_semihomogeneous}
Let $G = E_{-\lambda \underline{l}}$ be a simple semihomogeneous bundle for $\lambda \in \Q_{>0}$. 
Then
\begin{align*}
(h^0_{G,\underline{n}} (x) ,h^1_{G,\underline{n}} (x) ,h^2_{G,\underline{n}} (x) ) 
=\begin{cases}
(0,0,\chi_{G,\underline{n}} (x))   & \text{ if } x < s_{-}(G,\underline{n}) \\[1mm]
(0,-\chi_{G,\underline{n}} (x),0)   &  \text{ if } s_{-}(G,\underline{n})  < x < s_+(G,\underline{n}) \\[1mm]
(\chi_{G,\underline{n}} (x),0,0)   &  \text{ if } x > s_+(G,\underline{n}),
\end{cases}
\end{align*}
where 
\begin{align*}
 s_{\pm} (G,\underline{n}) \coloneqq \lambda   \frac{(\underline{l} \cdot \underline{n} ) \pm \sqrt{(\underline{l} \cdot \underline{n})^2 - (\underline{l}^2) (\underline{n}^2)}}{(\underline{n}^2)} .
\end{align*}
\end{lem}

\begin{proof}
Let $r=\rank G$.
Since $G$ is semihomogeneous, we have 
\begin{align*}
\chi_{G,\underline{n}} (x)  
=\frac{r}2 \cdot \left(-\lambda \underline{l} + x \underline{n}  \right)^2 =\frac{r}{2} \cdot \left((\underline{n}^2) x^2 - 2 \lambda(\underline{l} \cdot \underline{n}) x + \lambda^2 (\underline{l}^2) \right), 
\end{align*}
whose roots are $ s_{\pm}(G,\underline{n})$.
By \cite[Proposition 7.3]{MR498572},
there exists an isogeny $\pi \colon Y \to X$ and a line bundle $M$ on $Y$ such that $\pi^* G \simeq M^{\oplus r}$.
By \cite[Lemma 2.1.3]{Alvarado:2024aa},
we have
\begin{align}\label{eq_M}
h^i_{G,\underline{n}} (x) =\frac1{\deg \pi} h^i_{\pi^*G,\pi^*\underline{n}} (x)  = \frac{r}{\deg \pi} h^i_{M,\pi^*\underline{n}} (x) 
\end{align}
for any $i$.
Since  $ s_\pm(G,\underline{n})$ are the roots of the Hilbert polynomial $ \chi_{M,\pi^*\underline{n}} (x) = \frac{\deg \pi}{r}\chi_{G,\underline{n}} (x) $,
\begin{align*}
(h^0_{M,\pi^*\underline{n}} (x) ,h^1_{M,\pi^*\underline{n}} (x) ,h^2_{M,\pi^*\underline{n}} (x) ) 
=\begin{cases}
(0,0,\chi_{M,\pi^*\underline{n}} (x))   & \text{ if } x < s_{-}(G,\underline{n}) \\[1mm]
(0,-\chi_{M,\pi^*\underline{n}} (x),0)   &  \text{ if } s_{-}(G,\underline{n}) < x < s_+(G,\underline{n}) \\[1mm]
(\chi_{M,\pi^*\underline{n}} (x),0,0)   &   \text{ if } x > s_+(G,\underline{n})
\end{cases}
\end{align*}
by \cite[Example 4.1]{MR4157109}.
Hence this lemma holds by \ref{eq_M} and $\chi_{G,\underline{n}} (x)  = \frac{r}{\deg \pi} \chi_{M,\pi^*\underline{n}} (x) $.
\end{proof}

\begin{prop}\label{prop_h^i}
Assume that there exists an exact sequence\footnote{\ref{eq_resolution_I_p} is numerically consistent with Lemma \ref{lem_rank_chi}, which ensures that $v(E_{-\frac{2y_0}{x_0+1} \underline{l}} ) - v(E_{-\frac{2y_0}{x_0-1} \underline{l}} ) = v(\cali_o)$.}
\begin{align}\label{eq_resolution_I_p}
0 \to E_{-\frac{2y_0}{x_0-1} \underline{l}} \to E_{-\frac{2y_0}{x_0+1} \underline{l}} \to \cali_o \to 0.
\end{align}
If  $x_0^2 (\underline{n}^2) \geq 2 y_0^2 (\underline{l} \cdot \underline{n})^2$, then 
\begin{align*}
\beta^0(X,\underline{n}) &= s_+\Big(E_{-\frac{2y_0}{x_0+1} \underline{l}}, \underline{n} \Big) =  \frac{2y_0}{x_0+1}  \cdot \frac{(\underline{l} \cdot \underline{n} ) + \sqrt{(\underline{l} \cdot \underline{n})^2 - (\underline{l}^2) (\underline{n}^2)}}{(\underline{n}^2)}  ,\\[1mm]
\beta(X,\underline{n}) &= s_- \Big(E_{-\frac{2y_0}{x_0-1} \underline{l}},\underline{n} \Big) =  \frac{2y_0}{x_0-1} \cdot  \frac{(\underline{l} \cdot \underline{n} ) - \sqrt{(\underline{l} \cdot \underline{n})^2 - (\underline{l}^2) (\underline{n}^2)}}{(\underline{n}^2)} .
\end{align*}
\end{prop}

\begin{proof}
Set $E=E_{-\frac{2y_0}{x_0-1} \underline{l}}, F=E_{-\frac{2y_0}{x_0+1} \underline{l}}$ for simplicity.
We note that  $ s_+(F, \underline{n})  \leq s_-(E,\underline{n}) $ by the assumption $x_0^2 (\underline{n}^2) \geq 2 y_0^2 (\underline{l} \cdot \underline{n})^2$.
By the exact sequence  \ref{eq_resolution_I_p} and Lemma \ref{lem_h^i_semihomogeneous}, we have
\begin{align*}
h^0_{\cali_o,\underline{n}} (x)  &= 
\begin{cases}
0 &  \text{ if }  x <s_+(F, \underline{n}) \\
 \chi_{F,\underline{n}} (x)    &  \text{ if }  s_+(F, \underline{n})  < x < s_-(E,\underline{n}) \\
 \chi_{\cali_o,\underline{n}} (x) = \frac{(\underline{n}^2)}2 x^2-1 \hphantom{--} &  \text{ if }  x >  s_-(E,\underline{n}) 
\end{cases}
\end{align*}
and $h^1_{\cali_o,\underline{n}} (x)  =h^0_{\cali_o,\underline{n}} (x)- \frac{(\underline{n}^2)}2 x^2+1 $.
Hence this proposition holds.
\end{proof}

\begin{rem}
The condition $x_0^2 (\underline{n}^2) \geq 2 y_0^2 (\underline{l} \cdot \underline{n})^2$ is equivalent to $ \left( (\underline{l}^2) + \frac{1}{2y_0^2} \right) (\underline{n}^2) \geq (\underline{l} \cdot \underline{n})^2$,
which is satisfied if $\underline{n} \in \NS(X)$ is sufficiently close to being proportional to $\underline{l}$.
\end{rem}

To apply Proposition \ref{prop_h^i}, we need to show the existence of the exact sequence \ref{eq_resolution_I_p}.
In the following subsections, we show the existence in some cases.

\subsection{The case when $2d \mid (\underline{l} \cdot D)$ for every divisor $D$}

For a coherent sheaf $F $, let $\mu_{\underline{l}}(F) \coloneqq \frac{\underline{l} \cdot \ch_1(F)}{\underline{l}^2 \cdot \ch_0(F)}$ be the slope.
Let $D^b(X)$ be the bounded derived category of $X$.
For $\beta \in \R$, let
\begin{align*}
\calf_\beta &\coloneqq \{ F \in \Coh(X) \mid \mu_{\underline{l}}(S) \leq  \beta \text{ for any subsheaf } S \subset F \},\\
\calt_\beta &\coloneqq \{ F \in \Coh(X) \mid \mu_{\underline{l}}(Q) >  \beta \text{ for any quotient } F \twoheadrightarrow Q \},\\
\Coh^\beta(X) &\coloneqq \{ F \in D^b(X) \mid  \calh^{-1}(F) \in \calf_\beta, \calh^0(F) \in \calt_\beta, \calh^i(F) =0 \text{ for } i \neq 0,-1 \}.
\end{align*}
For each $(\alpha, \beta) \in \R_{>0} \times \R$, there exists a Bridgeland stability condition 
$\sigma_{\alpha, \beta} = (\Coh^\beta (X), Z_{\alpha,\beta})$ by \cite{MR2376815}, \cite{Arcara_2012}.
We refer the reader to \cite{MR3729077} and \cite[\S 2]{MR4579982} for Bridgeland stability.

\begin{prop}\label{prop_resolution_Rojas}
Assume that $2d \mid (\underline{l} \cdot D)$ for every divisor $D$ as in Theorem \ref{thm_Rojas}.
Then there exists an exact sequence $0 \to E_{-\frac{2y_0}{x_0-1} \underline{l}} \to E_{-\frac{2y_0}{x_0+1} \underline{l}} \to \cali_o \to 0 $. 
\end{prop}

\begin{proof}
By \cite[Proposition 14.2]{MR2376815}, \cite[Exercise 6.27]{MR3729077}, $\cali_o$ is $\sigma_{\alpha,\beta}$-semistable for $\alpha \gg 0$ and $\beta < 0$.
By Proposition \ref{prop_Rojas}, we have $\beta(X,\underline{l}) =\frac{2y_0}{x_0-1}$.
Combining this with  \cite[Lemma 3.1]{MR4363833} and the preceding paragraph therein,
$\cali_o$ is destabilized along a wall defined by an exact sequence $0 \to S \to \cali_o \to Q \to 0$ in $\Coh^{- \frac{\sqrt{d}}{d}} (X)$
such that
\[
v(S) = \left( \frac{x_0+1}2, \ch_1(S), \frac{x_0-1}{2} \right), \quad v(Q) = \left( -\frac{x_0-1}2, \ch_1(Q), -\frac{x_0+1}{2} \right)
\]
with $\ch_1(S) + \ch_1(Q) = 0$ and $\underline{l} \cdot \ch_1(S) = -2dy_0$.\footnote{In \cite{MR4363833}, $v(F)$ is defined by $v(F) = (\underline{l}^2 \cdot \ch_0(F), \underline{l} \cdot \ch_1(F), \ch_2(F) )$, which differs from ours.}
By  \cite[\S 7.2]{MR4579982}, $S, Q $ are $\sigma_{\alpha, \beta}$-semistable for any $\alpha >0$ and $-\frac{2y_0}{x_0-1} <\beta < -\frac{2y_0}{x_0+1}$.

Applying  \cite[Proposition 2.11 (2), Example 2.12 (1)]{MR4579982} to $Q \in  \Coh^{-\frac{2y_0}{x_0-1}} (X)$,
we find that $E \coloneqq Q[-1]$ is a  semihomogeneous bundle on $X$ since the base field is $\C$.
Since $E$ is semihomogeneous with $v(E) =-v(Q) = \left( \frac{x_0-1}2, \ch_1(S), \frac{x_0+1}{2}  \right)$, we have
\[
\ch_1(S)^2 = 2 \cdot \frac{x_0-1}2 \cdot \frac{x_0+1}{2}  =2dy_0^2 = \frac{(\underline{l} \cdot \ch_1(S))^2 }{\underline{l}^2 }.
\]
Hence $\ch_1(S)=\ch_1(E)$ is proportional to $\underline{l}$.
Thus
$E=E_{-\frac{2y_0}{x_0-1} \underline{l}} $
by $\rank E=\frac{x_0-1}{2}$ and $\underline{l} \cdot \ch_1(E)= -2dy_0$.

On the other hand, $S$ is a Gieseker semistable torsion-free sheaf with respect to $\underline{l}$ by \cite[Exercise 6.27]{MR3729077}.
Hence $S[1] \in \Coh^{-\frac{2y_0}{x_0+1}} (X)$ since  $S$ is $\mu_{\underline{l}}$-semistable with  $\mu_{\underline{l}} (S) = -\frac{2y_0}{x_0+1}$.
Then $S$ is semihomogeneous by \cite[Proposition 2.11 (2), Example 2.12 (1)]{MR4579982} again and hence  $S=E_{-\frac{2y_0}{x_0+1} \underline{l}} $.
Since $0 \to S \to \cali_o \to E[1] \to 0$ induces $0 \to E \to S \to \cali_o \to 0$, this completes the proof.
\end{proof}

\begin{proof}[Proof of Theorem \ref{thm_BFT_surface}]
This theorem follows from Propositions \ref{prop_h^i} and \ref{prop_resolution_Rojas}.
\end{proof}

\begin{rem}
In \cite[page.~772]{MR4518246}, the author asks whether $\beta(\un_1)^{-1} + \beta(\un_2)^{-1} \leq \beta(\un_1+\un_2)^{-1} $ holds for arbitrary polarizations $\un_1,\un_2$.
We see that the answer to this question is negative. 
In fact, for $\un$ satisfying the assumption in Theorem \ref{thm_BFT_surface},  it holds that
\begin{align*}
\beta(X,\un)^{-1} = \frac{x_0-1}{2y_0} \cdot   \frac{(\underline{l} \cdot \underline{n} ) + \sqrt{(\underline{l} \cdot \underline{n})^2 - (\underline{l}^2) (\underline{n}^2)}}{(\underline{l}^2)} .
\end{align*}
If we take $\un_1  = m \ul + \underline{c}, \un_2=m \ul  -\underline{c}$ for a fixed $\underline{c} \in \NS(X)$ not proportional to $\underline{l}$ and an integer $m \gg 0$,
then $\un_1,\un_2$ are polarizations satisfying the assumption in Theorem \ref{thm_BFT_surface}.
A direct computation shows that $\beta(\un_1)^{-1} + \beta(\un_2)^{-1} >  \beta(\un_1+\un_2)^{-1} $ holds in this case.
For an explicit example, see Example \ref{ex_d=2}.
\end{rem}

\subsection{The case $d \leq 6$}

In this subsection, 
we show the following proposition, which weakens the assumptions of Proposition \ref{prop_resolution_Rojas} for $d \leq 6$.

\begin{prop}\label{prop_d=2,6}
Assume that $d \in \{2,3,5,6\}$ and there is no elliptic curve $C$ on $X$ with $(\underline{l} \cdot C )  < \frac{x_0-1}{y_0}$.
Then there exists an exact sequence $0 \to E_{-\frac{2y_0}{x_0-1} \underline{l}} \to E_{-\frac{2y_0}{x_0+1} \underline{l}} \to \cali_o \to 0 $.
\end{prop}

To simplify the notation, we set
$
E \coloneqq E_{-\frac{2y_0}{x_0-1} \underline{l}} ,  F \coloneqq E_{-\frac{2y_0}{x_0+1} \underline{l}}
$.
Then we have
\begin{align*}
v(E) = \left( \frac{x_0-1}{2}, \, -y_0 \underline{l}, \, \frac{x_0+1}{2} \right) , \quad v(F) = \left( \frac{x_0+1}{2}, \,  -y_0 \underline{l}, \,  \frac{x_0-1}{2} \right) 
\end{align*}
by Lemma \ref{lem_rank_chi}.
By the Hirzebruch--Riemann--Roch theorem,
we have
\begin{align*}
\chi(E^\vee \otimes F) = \left(\frac{x_0-1}{2} \right)^2 -y_0^2 \cdot  \underline{l}^2 +\left(\frac{x_0+1}{2} \right)^2 = \frac{x_0^2 +1 - 4dy_0^2}{2} =1.
\end{align*}
Since $E^\vee \otimes F$ is semihomogeneous and $\ch_1 (E^\vee \otimes F) = y_0 \underline{l} $ is ample,
we have $h^i(E^\vee \otimes F)=0$ for $i >0$ by \cite[Proposition 2.2.1]{Alvarado:2024aa}.
Hence $\dim \Hom (E,F) =1$ and there exists a non-zero homomorphism
\[
f \colon E =E_{-\frac{2y_0}{x_0-1} \underline{l}}  \to F =E_{-\frac{2y_0}{x_0+1} \underline{l}} .
\]

\begin{lem}\label{lem_resolution_I_p}
\begin{enumerate}
\item Assume that $f$ is not injective and set $r\coloneqq \rank (\im f)$,  $\underline{h} \coloneqq -  \ch_1 (\im f) \in \NS(X)$.
Then 
\[
\frac1r \underline{h} -\frac{2 y_0}{x_0+1}  \underline{l}, \qquad \frac{2 y_0}{x_0-1}  \underline{l}-\frac{1}r\underline{h}
\]
are nef and 
\[
\frac{x_0-1}{y_0} < \frac{(\underline{l} \cdot \underline{h})}{r} < \frac{x_0+1}{y_0} .
\]
\item Assume that $f$ is injective and $\coker f$ has torsion.
Then there exists an elliptic curve $C \subset X$  with 
\[
(\underline{l} \cdot C)  < \frac{x_0-1}{y_0}.
\]
\item Assume that $f$ is injective and $\coker f$ is torsion-free.
Then $\coker f \simeq P \otimes \cali_p$ for some $P \in \widehat{X}$  and $p \in X$.
\end{enumerate}
\end{lem}

\begin{proof}
(1) Since $(\underline{l} \cdot \ch_1(E))/\rank E = -\frac{x_0+1}{y_0}$ and $(\underline{l} \cdot \ch_1(F))/\rank F  = -\frac{x_0-1}{y_0}$,
the inequalities follow from the $\mu_{\underline{l}}$-stability of $E$ and $F$.

To show the nefness of $\frac1r \underline{h} -\frac{2 y_0}{x_0+1}  \underline{l}$, consider the homomorphism $g \colon F^\vee \to (\im f)^\vee$ between vector bundles.
Since $g$ is generically surjective,  $\wedge^r g \colon \wedge^r F^\vee \to \det (\im f)^\vee \eqqcolon H$ is non-zero.
By  \cite[Proposition 7.3]{MR498572},
there exists an isogeny $\pi \colon Y \to X$ and a line bundle $M$ on $Y$ such that $\pi^* F^\vee \simeq M^{\oplus \rank F}$.
We note that $\ch_1(M) = \frac{\ch_1(\pi^* F^\vee)}{\rank F} =\frac{2 y_0}{x_0+1} \pi^* \underline{l}$.
Then we have a non-zero map $\pi^* (\wedge^r g) \colon \wedge^r (M^{\oplus \rank F}) \to \pi^* H$.
In particular, there exists a non-zero map $M^{\otimes r} \to \pi^*H$ and hence
$\ch_1(\pi^* H) - r \ch_1(M) = \pi^* (\underline{h} -\frac{2 y_0 r}{x_0+1}  \underline{l})$ is nef.
Thus so is $\underline{h} -\frac{2 y_0 r}{x_0+1}  \underline{l}$.

By considering $E \to \im f$,  we see that $\frac{2 y_0}{x_0-1}  \underline{l} - \frac1r \underline{h} $ is nef similarly.
\\
(2) Let $T =(\coker f)_{\mathrm{tor}}$ be the torsion part.
By the snake lemma, we have an exact sequence
\[
0 \to E \to E' \to T \to 0,
\]
where $E'$ is the kernel of $F \to (\coker f)/T$.
If $T$ is a skyscraper sheaf, then $\Ext^1( T, E) =\Ext^1(E,T)^\vee= H^1(E^\vee\otimes T)^\vee=0$ and hence $ E' \simeq T \oplus E$,
which contradicts the torsion-freeness of $E' \subset F$.
Thus $\dim \Supp T =1$.

Then $\ch_1(E')=\ch_1(E) + \ch_1(T) =-y_0 \underline{l} +[D]$ for some non-zero effective divisor $D$.
By the $\mu_{\underline{l}}$-stability of $F$, we have
\[
\frac{\underline{l} \cdot \ch_1(E')}{\rank E'}=\frac{-2dy_0 +(\underline{l} \cdot D)}{\frac{x_0-1}{2}} < \frac{\underline{l} \cdot \ch_1(F)}{\rank F} = \frac{-2dy_0}{\frac{x_0+1}{2}}
\]
that is,
\[
(\underline{l} \cdot D) < \frac{4dy_0}{x_0+1} =\frac{x_0-1}{y_0}.
\]
By the Hodge index theorem,
we have
\begin{align*}
(D^2) \leq  (\underline{l} \cdot D)^2/(\underline{l}^2) < \frac{(x_0-1)^2}{2dy_0^2} =\frac{2(x_0-1)}{x_0+1}<2
\end{align*}
and hence $(D^2)=0$ by $(D^2) \in 2\Z_{\geq 0}$.
Then any irreducible component $C$ of $D$ is an elliptic curve with $(\underline{l} \cdot C) \leq (\underline{l} \cdot D) < \frac{x_0-1}{y_0}.$\\
(3) 
Since $\coker f$ is a torsion-free sheaf of rank one,
we have $\coker f \simeq \det F \otimes \det E^\vee \otimes \cali_Z$ for some zero-dimensional subscheme $Z \subset X$.
Since $\ch_1(E)=\ch_1(F) \in \NS(X)$,  the line bundle $ \det F \otimes \det E^\vee $ is numerically trivial.
Since $\chi(\coker f) =\chi(F) -\chi(E) =- 1$, the length of $Z$ is one.
\end{proof}

\begin{lem}\label{lem_trivial_solution}
Let $r$ be a positive integer and $\uh \in \NS(X)$ satisfying
\begin{align}\label{eq_trivial_solution0}
r=y_0, \quad (\underline{l}\cdot \underline{h}) =x_0, \quad (\underline{h}^2)=2y_0^2.
\end{align}
Then there exists an elliptic curve $C$ on $X$ such that $(\ul \cdot C) < \frac{x_0-1}{y_0}$.
\end{lem}

\begin{proof}
Since $(t \ul-\uh)^2=2 d t^2 -2 x_0 t + 2y_0^2= 2d (t- \frac{x_0+1}{2d})(t-\frac{x_0-1}{2d})$,
\begin{align*}
\frac{x_0+1}{2d} \ul -\uh , \quad \uh - \frac{x_0-1}{2d} \ul 
\end{align*}
are nef with self-intersection number zero.
Let $A=\gcd(\frac{x_0+1}{2}, d) , \ B=\gcd(\frac{x_0-1}{2}, d) $.
Since $\frac{x_0+1}{2}, \frac{x_0-1}{2} $ are coprime
and $\frac{x_0+1}{2} \cdot \frac{x_0-1}{2} = d y_0^2$ is divisible by $d$,
we have $AB=d$.

Since any nef divisor on an abelian variety is algebraically equivalent to an effective divisor by \cite[Lemma 1.1]{MR1646050},
\begin{align}\label{eq_trivial_solution}
\frac{x_0+1}{2A} \ul - \frac{d}{A}\uh =s [C] , \qquad \frac{d}{B}\uh - \frac{x_0-1}{2B} \ul =s'[C'] 
\end{align}
for some elliptic curves $C,C'$ and positive integers $s, s'$.
Since
\[
ss'(C \cdot C') = \frac1{AB}\left( \left(\frac{x_0+1}{2} \ul - d \uh \right) \cdot  \left(d\uh - \frac{x_0-1}{2} \ul  \right) \right) =\frac{d(x_0^2-4dy_0^2)}{AB} =1,
\]
we have $s=s'=(C \cdot C')=1$.
Then  $ \underline{l} = A[C] + B[C']$ by \ref{eq_trivial_solution} and hence
\[
\min \{(\ul \cdot C), (\ul \cdot C') \} =\min\{B,A\} \leq \sqrt{d} =\frac{\sqrt{x_0^2-1}}{2y_0} < \frac{x_0-1}{y_0},
\]
where the last inequality holds by $x_0 \geq 3$.
\end{proof}

\begin{proof}[Proof of Proposition \ref{prop_d=2,6}]
If $f \colon E \to F $ is injective and $\coker f$ is torsion-free,
then by replacing $E,F$ with suitable twists $t_p^* E \otimes P', t_p^* F \otimes P''$ for some $p\in X$ and $P',P'' \in \widehat{X}$,
we obtain $0 \to E \to F \to \cali_o \to 0$ from Lemma \ref{lem_resolution_I_p} (3).
Hence it suffices to show that $f$ is injective and $\coker f$ is torsion-free.
Since we assume that there is no elliptic curve $C$ with $(\ul \cdot C) < \frac{x_0-1}{y_0}$,
it suffices to show that $f$ is injective by Lemma \ref{lem_resolution_I_p} (2).

Assume that $f$ is not injective.
By Lemma \ref{lem_resolution_I_p} (1), $r=\rank (\im f) $ and $\uh=-\ch_1(\im f)$ satisfy
\begin{align*}
1 &\leq r < \rank E=\frac{x_0-1}{2}\\
\frac{x_0-1}{y_0} &< \frac{(\underline{l} \cdot \underline{h})}{r} < \frac{x_0+1}{y_0} ,\\
2d (\uh^2) &= (\ul)^2 \cdot (\uh)^2 \leq (\ul \cdot \uh)^2,\\
\left(\frac1r \underline{h} -\frac{2 y_0}{x_0+1}  \underline{l} \right)^2 &= \frac{(\uh^2)}{r^2} -\frac{4y_0}{(x_0+1)r} (\ul \cdot \uh) + \frac{2(x_0-1)}{x_0+1} \geq 0,\\
 \left( \frac{2 y_0}{x_0-1}  \underline{l}-\frac{1}r\underline{h} \right)^2 &= \frac{(\uh^2)}{r^2} -\frac{4y_0}{(x_0-1)r} (\ul \cdot \uh) + \frac{2(x_0+1)}{x_0-1} \geq 0.
\end{align*}

If $d=2$, there is no such $r$ since $x_0=3$.
If $d =3,5,6$, then $(x_0,y_0) =(7,2), (9,2), (5,1)$, respectively. 
Then we can check that $r, (\ul \cdot \uh), (\uh^2)$ satisfy \ref{eq_trivial_solution0} and hence  there exists an elliptic curve $C$ on $X$ such that $(\ul \cdot C) < \frac{x_0-1}{y_0}$ by Lemma \ref{lem_trivial_solution}, which contradicts the assumption.
Hence $f$ is injective and this proposition is proved.
\end{proof}

\begin{rem}
If $d=7$, then $(x_0,y_0)=(127,24)$ and hence $r=29 \neq y_0, \, \uh = 11 \ul $ satisfy the numerical conditions in Lemma \ref{lem_resolution_I_p} (1).
Thus, the argument for the injectivity of $f$ in the proof of Proposition \ref{prop_d=2,6} does not work for $d=7$.
\end{rem}

\begin{ex}\label{ex_d=2}
Let $X$ be an abelian surface such that there exist elliptic curves $C_1,C_2 $ on $X$ with $\NS(X)=\Z [C_1] \oplus \Z[C_2]$ and $(C_1 \cdot C_2)=2$.
For the class $\un$ of $N=a C_1+bC_2$ with $a,b >0$,
we can determine $\beta(\un) $
as follows.

Let $\ul=[C_1] + [C_2]$, whose type  is $(1,2)$.
Since the minimal solution to  $x^2-8d y^2=1$  is  $(x_0,y_0)=(3,1)$ and  there is no elliptic curve $C$ with $(\ul \cdot C ) < \frac{x_0-1}{y_0} =2$,
there exists an exact sequence $0 \to E_{- \underline{l}} \to E_{-\frac12 \underline{l}} \to \cali_o \to 0 $ by Proposition \ref{prop_d=2,6}.
Hence it holds that 
\begin{align*}
 \beta(\un) = \frac{2 \cdot 1}{3-1} \cdot \frac{2(a+b) - \sqrt{4(a+b)^2 -16ab} }{4ab} 
 =\begin{cases}
 \frac1{b} & \text{ if } \ \frac12 b \leq  a \leq b \\
 \frac1a &  \text{ if }  \ b \leq  a \leq 2b
\end{cases}
\end{align*}
for $ \frac12 b \leq  a \leq 2b$ by Proposition \ref{prop_h^i}.

For $ a \geq 2b$, we consider an exact sequence $0 \to \calo_X(-C_1) \to \cali_o \to \cali_{o/C_1} \to 0$.
Since $-C_1+ x N$ is ample if $x > \frac1{a}$, 
\begin{align*}
(h^0_{\calo_X(-C_1),\underline{n}} (x) ,h^1_{\calo_X(-C_1),\underline{n}} (x) , h^2_{\calo_X(-C_1),\underline{n}} (x)  ) =(\chi_{\calo_X(-C_1),\underline{n}} (x),0,0) 
\end{align*}
holds for $x > \frac1{a}$.
On the other hand, $\calo_{C_1}(-o) + x N|_{C_1} =\cali_{o/C_1} + x N|_{C_1}$ on $C_1$ is ample if and only if $ x > \frac1{2b}$.
Hence 
\begin{align*}
(h^0_{ \cali_{o/C_1} ,\underline{n}} (x), h^1_{ \cali_{o/C_1} ,\underline{n}} (x) , h^2_{ \cali_{o/C_1} ,\underline{n}} (x) ) 
=\begin{cases}
(0,-\chi_{ \cali_{o/C_1} ,\underline{n}} (x), 0)   & \text{ if }  \ x <  \frac1{2b} \\[1mm]
(\chi_{ \cali_{o/C_1} ,\underline{n}} (x), 0,0)   &  \text{ if } \ x > \frac1{2b}
\end{cases}
\end{align*}
holds. Thus we have $\beta(\un) =\frac1{2b}$ for $a \geq 2b$.
Similarly, $\beta(\un) =\frac1{2a}$ holds for $b \geq 2a$.

Hence we obtain
\begin{align*}
 \beta(\un) =\begin{cases}
  \frac1{2a} & \text{ if }  \ a \leq \frac12 b  \\[1mm]
 \frac1{b} & \text{ if }  \ \frac12 b \leq  a \leq b \\[1mm]
 \frac1a &  \text{ if } \  b \leq  a \leq 2b\\[1mm]
  \frac1{2b} & \text{ if } \  2b \leq a .
\end{cases}
\end{align*}
\end{ex}

\section{Proof of Corollary \ref{cor_irrational_Intro}}\label{sec_proof_cor}

\begin{proof}[Proof of Corollary \ref{cor_irrational_Intro}]
If 
\begin{align}\label{eq_4,34,12}
(\underline{l}^2)=4, \quad (\underline{n}^2)=34, \quad (\underline{l}\cdot \underline{n}) =12,
\end{align}
then $(x_0,y_0)=(3,1)$ and
\[
 \frac{2y_0}{x_0-1} \cdot  \frac{(\underline{l} \cdot \underline{n}) -\sqrt{(\underline{l} \cdot \underline{n})^2 -(\underline{l}^2) (\underline{n}^2)}}{(\underline{n}^2)}= \frac{6-\sqrt{2}}{17}.
\]
Moreover, the  inequality  $x_0^2 (\underline{n}^2) \geq 2 y_0^2 (\underline{l} \cdot \underline{n})^2$ holds.
Thus, by Propositions \ref{prop_h^i} and \ref{prop_resolution_Rojas} (or  \ref{prop_d=2,6} for $d=2$),  it suffices to show the existence of $X,\underline{l}$, and $\underline{n}$ satisfying \ref{eq_4,34,12} and $\NS(X) = \Z \underline{l} \oplus \Z\underline{n}$.

Let $ \Lambda $ be a lattice of rank six with a basis $\{e_1,\dots,e_6\}$ such that
\begin{align*}
(e_1 \cdot e_2)=(e_3\cdot e_4)=(e_5 \cdot e_6) =1
\end{align*}
and $(e_i \cdot e_j)=0$ otherwise.
Let $\Lambda' = \Z (e_1+e_2) \oplus \Z(e_3-2e_4) \subset \Lambda$, which is a lattice of signature $(1,1)$.
For  general $\omega$ in 
\[
\{p\in \Lambda \otimes_{\Z} \C \mid (p^2)=0, (p \cdot \bar{p}) >0 , (p \cdot (e_1+e_2)) =(p \cdot (e_3-2e_4))=0\},
\]
we have $\omega^{\perp} \cap \Lambda =\Lambda'$.
By the surjectivity of the period map for abelian surfaces \cite[Theorem II]{MR480530},
there exists an abelian surface $X$ with $\NS(X) \simeq  \omega^{\perp} \cap \Lambda =\Lambda'$.
Let $\underline{h}_1,\underline{h}_2 \in \NS(X) $ be the elements corresponding to $e_1+e_2,e_3-2e_4 \in \Lambda'$, respectively.
We have $(\underline{h}_1^2) = 2, (\underline{h}_2^2)=-4, (\underline{h}_1 \cdot \underline{h}_2)=0$.
Then $\underline{l}= 2\underline{h}_1+\underline{h}_2$ and $\underline{n}=5\underline{h}_1+2\underline{h}_2$ satisfy \ref{eq_4,34,12}.
Since either $\underline{l}$ or $-\underline{l}$ is ample, we may assume that $\underline{l}$ is ample by replacing $\underline{h}_1,\underline{h}_2$ with $-\underline{h}_1,-\underline{h}_2$ if necessary.
Then $(\underline{n}^2) >0$ and $(\underline{l} \cdot \underline{n}) >0$ imply that $\underline{n}$ is also  ample.
Since $\Z\underline{l} \oplus \Z \underline{n} = \Z\underline{h}_1 \oplus \Z \underline{h}_2 =\NS(X)$, this corollary is proved.
\end{proof}

\bibliographystyle{amsalpha}
\providecommand{\bysame}{\leavevmode\hbox to3em{\hrulefill}\thinspace}
\providecommand{\MR}{\relax\ifhmode\unskip\space\fi MR }
\providecommand{\MRhref}[2]{%
  \href{http://www.ams.org/mathscinet-getitem?mr=#1}{#2}
}
\providecommand{\href}[2]{#2}

\end{document}